\documentclass[reqno,11pt]{amsart}
\usepackage{amsmath,amssymb,latexsym,cancel,rotating,hyperref}

\usepackage{eucal}

\usepackage{amscd}
\usepackage[dvips]{color}
\usepackage{multicol}
\usepackage[all]{xy}           
\usepackage{graphicx}
\usepackage{color}
\usepackage{colordvi}
\usepackage{xspace}
\usepackage{tikz}

\usepackage{enumitem}

\numberwithin{equation}{section}

\newtheorem{theorem}{Theorem}[section]
\newtheorem{proposition}[theorem]{Proposition}

\newtheorem{lemma}[theorem]{Lemma}

\newtheorem{remark}[theorem]{Remark}
\newtheorem{example}[theorem]{Example}
\newtheorem{definition}[theorem]{Definition}

\input xy
\xyoption{poly}
\xyoption{2cell}
\xyoption{all}

\def\ZZ{\mathbb Z}

\def\ZZ{\mathbb{Z}}

\renewcommand{\eqref}[1]{{\rm (\ref{#1})}}

\begin{document}

\title[Some properties of generalized  cluster algebras of geometric types]
{Some properties of generalized  cluster algebras of geometric types}

\author{Junyuan Huang, Xueqing Chen, Fan Xu and Ming Ding}
\address{School of Mathematics and Information Science\\
Guangzhou University, Guangzhou 510006, P.R.China}
\email{jy4545@e.gzhu.edu.cn (J.Huang)}
\address{Department of Mathematics,
 University of Wisconsin-Whitewater\\
800 W. Main Street, Whitewater, WI.53190. USA}
\email{chenx@uww.edu (X.Chen)}
\address{Department of Mathematical Sciences\\
Tsinghua University\\
Beijing 100084, P.~R.~China} \email{fanxu@mail.tsinghua.edu.cn (F.Xu)}
\address{School of Mathematics and Information Science\\
Guangzhou University, Guangzhou 510006, P.R.China}
\email{dingming@gzhu.edu.cn (M.Ding)}





\keywords{generalized  cluster algebra, cluster variable, generalized projective cluster variable, standard monomial}

\maketitle

\begin{abstract}
We study the  lower bound algebras generated by the generalized projective cluster variables of acyclic generalized cluster algebras of geometric types.  We prove that this lower bound algebra coincides with the corresponding generalized cluster algebra under the coprimality condition.   As a corollary, we  obtain  the dual PBW bases of these generalized cluster algebras. Moreover, we show that if the standard monomials of a generalized cluster algebra of geometric type are linearly independent, then the directed graph associated to the initial generalized seed of this generalized cluster algebra does not have 3-cycles.
\end{abstract}


\section{Background}
Fomin and Zelevinsky invented cluster algebras \cite{ca1,ca2} in order
to give an algebraic framework for studying the theory of total positivity and canonical bases
in quantum groups  developed by Lusztig.
 Berenstein, Fomin and
Zelevinsky  defined the lower bound in~\cite{bfz} which is a subalgebra of the associated cluster algebra generated by specific finitely many cluster variables with better behavior. They
proved that a cluster algebra coincides with
the lower bound  if and only if it  possesses an~acyclic seed. Under this acyclicity condition, the standard monomial bases of~these kinds of~cluster algebras can be naturally constructed.

Recently, Baur and Nasr-Isfahani \cite{BN} introduced the notion of another lower bound cluster algebra which is generated by the initial cluster variables and projective cluster variables. They showed that, under an
assumption of “acyclicity”,  a cluster algebra coincides with its  lower bound cluster algebra, and further constructed a basis for this cluster algebra.  Note that Qin pointed out that this basis should agree with an associated dual
PBW basis in the language of \cite{KQ}, which plays an important role in defining the triangular basis \cite{Q1}\cite{Q2}.

Chekhov and Shapiro introduced the notion of  generalized cluster algebras  \cite{CS}  in order to study the Teichm\"{u}ller spaces of~Riemann surfaces with holes and orbifold points.
In this generalization, the hallmark binomial exchange relations of cluster algebras  are replaced with polynomials with arbitrarily many terms in
generalized cluster algebras.  Generalized cluster algebras also possess the Laurent phenomenon \cite{CS}  and  are studied in a similar way as cluster algebras (see for example \cite{CL1, CL2, Mou, nak2}).
Gekhtman, Shapiro and Vainshtein \cite{gsv18} proved that generalized upper cluster algebras coincide with upper bounds under certain coprimality conditions, which has been extended by Bai, Chen, Ding and Xu \cite{BCDX-2} in the quantum setting. Later, Bai, Chen, Ding and Xu \cite{BCDX-1} proved that the conditions of acyclicity
and coprimality close the gap between lower bounds and upper bounds associated to generalized
cluster algebras, and consequently obtained the standard monomial bases of these  algebras.
Recently, Du and Li \cite{DL} proved that for the Laurent phenomenon algebra, the above results also hold under  certain conditions.

The aim of this paper is a continuation to investigate the structure theory of generalized cluster algebras. We will firstly construct the dual PBW bases of acyclic generalized cluster algebras of geometric types under the coprimality condition;
then we show that if the standard monomials of a generalized cluster algebra of geometric type are linearly independent, the directed graph associated to the initial generalized seed of this generalized cluster algebra contains no 3-cycles.

\section{Preliminaries}

We recall some terminology,  definitions and related results of~the generalized cluster algebras of~geometric types.

Let $i, j$ be positive integers satisfying $i<j$,  write $[i,j]$ for the set $\{i,i+1,\ldots,j-1,j\}$.
A square integer  matrix $B$ is called skew-symmetrizable if there is a positive integer
diagonal matrix $D$ such that $DB$ is skew-symmetric, where $D$ is said to be the skew-symmetrizer of
$B$.
Let $\widetilde{B}=(b_{ij})$ be an~$m\times n$ integer matrix for some $m\geq n$. The principal part $B$ of~$\widetilde{B}$ is defined to be  its upper $n\times n$ submatrix.
For each $i \in [1,n]$, let $d_i$ be a positive integer such that $d_i$ divides $b_{ji}$ for any $j\in[1,n]$, and  set $\beta_{ji}=\frac{b_{ji}}{d_i}$ for any $j\in[1,m]$.

Denote by $\mathbb{Q}(x_{1},x_2,\ldots,x_{m})$ the function field of~$m$ variables over $\mathbb{Q}$ with a transcendence basis $\{x_1,x_2,\ldots,x_m\}$.
Let $\mathbb{P}$ be the coefficient group which is  the multiplicative free abelian group generated by $\{x_{n+1},\ldots,x_{m} \}$ and $\mathop{\mathbb{ZP}}$ be its integer group ring.
\begin{definition}\label{def of generalized seed}
With above notations, a generalized seed is defined to be a triple $\big(\widetilde{\mathbf{x}},\rho,\widetilde{B}\big)$, where
\begin{itemize}
 \item[(1)] the set $\widetilde{\mathbf{x}}=\{x_{1},x_2,\ldots, x_{m}\}$ is~called the extended cluster, $\mathbf{x}=\{x_{1},x_2,\ldots,x_{n}\}$ is~called the cluster,
  elements $\{x_1, x_2, \ldots,x_{n}\}$ are called cluster variables and elements $\{x_{n+1},x_2,\ldots,x_{m}\}$ are called frozen variables;
 \item[(2)] $\rho=\{\rho_{1},\rho_2,\ldots,\rho_{n}\}$ is the set of strings, for each $i\in [1,n]$ the $i$-th string is denoted by $\rho_i:=\{\rho_{i,0},\ldots,\rho_{i,d_i}\}$ with $\rho_{i,0}=\rho_{i,d_i}=1$ and $\rho_{i,j}$ being monomials in~$\mathbb{Z} [x_{n+1},\ldots,x_{m}]$ for $j \in [1,d_i-1]$;
\end{itemize}
\end{definition}

Denote the floor function for real numbers $x$ by  $\lfloor x\rfloor$. Define the function
\begin{gather*}
[x]_+=
 \begin{cases}
 x,  &\text{if}\quad x\geq 0;
 \\
0,  & \text{if}\quad x< 0.
 \end{cases}
\end{gather*}

\begin{definition}
For $i\in [1,n]$, the mutation of~a generalized seed $\big(\widetilde{\mathbf{x}},\rho,\widetilde{B}\big)$ in the direction $i$ is another generalized seed \smash{$\mu_i\big(\widetilde{\mathbf{x}},\rho,\widetilde{B}\big):=\big(\widetilde{\mathbf{x}}',
\rho',\widetilde{B}'\big)$}, where
\begin{itemize}\itemsep=0pt
\item[(1)] the set $\widetilde{\mathbf{x}}':=(\widetilde{\mathbf{x}}- \{x_i\})\cup\{x'_{i}\}$ with
 \begin{gather*}
 x_{i}':=\mu_i(x_i)=x_i^{-1}\Big(\sum\limits_{r=0}^{d_i}\rho_{i,r}\prod\limits_{j=1}^{m} x_{j}^{\lfloor[r\beta_{ji}]_+\rfloor +\lfloor[-(d_i-r)\beta_{ji}]_+\rfloor}\Big)\!,
 \end{gather*}
 which is called the exchange relation;
\item[(2)] the set $\rho' :=\mu_i(\rho)=(\rho- \{\rho_i\})\cup\{\rho'_{i}\}$, where $\rho_{i}'=\{\rho'_{i,0},\ldots,\rho'_{i,d_i}\}$ such that $\rho'_{i,r}=\rho_{i,d_i-r}$ for~$0\leq r\leq d_i$;

\item[(3)]
the matrix $\widetilde{B}':=\mu_i\big(\widetilde{B}\big)$ is defined by
\begin{gather*}
b'_{kl}=
 \begin{cases}
 -b_{kl}, &\text{if}\quad k=i\quad \text{or}\quad l=i;
 \\
b_{kl}+\displaystyle\frac{|b_{ki}|b_{il}+b_{ki}|b_{il}|}{2}, &\text{otherwise}.
 \end{cases}
\end{gather*}
\end{itemize}
\end{definition}

Note that $\mu_i$ is an involution.  Two generalized seeds are called mutation-equivalent if one can be obtained from
another by a sequence of mutations.

\sloppy\begin{definition}\label{def of gca}
Given an initial generalized seed $\big(\widetilde{\mathbf{x}}, \rho, \widetilde{B}\big)$, the generalized cluster algebra $\smash{\mathcal{A}\big(\widetilde{\mathbf{x}}, \rho, \widetilde{B}\big)}$ is the $\mathop{\mathbb{ZP}}$-subalgebra of~$\mathbb{Q}(x_1,\ldots,x_m)$ generated by all cluster variables from all gene\-ralized seeds which are mutation-equivalent to $\smash{\big(\widetilde{\mathbf{x}},\rho,\widetilde{B}\big)}$.
\end{definition}
Note that  one can recover the classical cluster algebras by setting $d_i=1$ for all $i \in [1,n]$.

For each $i\in [1,n]$,  we know that
\begin{gather*}
 x_{i}x'_{i}=\sum\limits_{r=0}^{d_i}\rho_{i,r}\prod\limits_{j=1}^{m} x_{j}^{\lfloor[r\beta_{ji}]_+\rfloor +\lfloor[-(d_i-r)\beta_{ji}]_+\rfloor}\in \mathop{\mathbb{ZP}}[x_1,\ldots,x_{i-1},x_{i+1},\ldots,x_n].
\end{gather*}

\begin{definition}
A generalized seed $\big(\widetilde{\mathbf{x}},\rho,\widetilde{B}\big)$ is called coprime if $x_{i}x'_{i}$ and $x_{j}x'_{j}$ are coprime in $\mathbb{ZP}[x_1,\ldots,x_n]$ for~any two different $i$, $j\in [1,n]$.
\end{definition}

The directed graph associated to a generalized seed~$\big(\widetilde{\mathbf{x}},\rho,\widetilde{B}\big)$ is denoted by   $\Gamma\big(\widetilde{\mathbf{x}},\rho,\widetilde{B}\big)$ with vertices $[1,n]$ and the directed edges from $i$ to $j$ if $b_{ij}>0$.

\begin{definition}
 A generalized seed $\big(\widetilde{\mathbf{x}},\rho,\widetilde{B}\big)$ is called acyclic if $\Gamma\big(\widetilde{\mathbf{x}},\rho,\widetilde{B}\big)$ does not contain any oriented cycle.
A~gene\-ra\-li\-zed cluster algebra is called acyclic if it has an~acyclic generalized~seed.
\end{definition}

Chekhov and Shapiro prove the  Laurent phenomenon on generalized  cluster algebras~\cite{CS}, i.e., the generalized cluster algebra $\mathcal{A}\big(\widetilde{\mathbf{x}},\rho,\widetilde{B}\big)$ is a $\ZZ\mathbb{P}$-subalgebra of the ring of  Laurent polynomials in the cluster variables from any cluster.

In the setting of generalized cluster algebras, we can similarly define the notion of standard monomials which is a natural generalization of one in \cite{bfz}.
\begin{definition}\cite{BCDX-1}
Let~$\big(\widetilde{\mathbf{x}},\rho,\widetilde{B}\big)$ be a generalized seed. A~standard monomial in ${x_1},{x'_1}, \ldots ,{x_n},{x'_n}$ is a monomial that contains no product of~the form $x_ix'_i$ for any $i\in [1,n]$.
\end{definition}

The following theorem, which is parallel to the corresponding
result in \cite{bfz}, will be useful for us to  construct the dual PBW bases for generalized cluster algebras under certain conditions.
\begin{theorem}\cite{BCDX-1}\label{standard-1}
Let  $( {\widetilde{\mathbf{x}},\rho,{\widetilde B}})$ is an acyclic and coprime generalized seed, then the set of standard monomials in ${x_1},{x'_1}, \ldots ,{x_n},{x'_n}$ is a $\ZZ\mathbb{P}$-basis of  the generalized cluster algebra $\mathcal{A}\big(\widetilde{\mathbf{x}},\rho,\widetilde{B}\big)$.
\end{theorem}

\section{The dual PBW bases}

Up to simultaneously reordering of columns and rows, we  assume  that the entries in the skew-symmetrizable matrix $B$ satisfy  $b_{ij} \geq 0$ for any $i>j$.
In the following, we let $\Sigma=( {\widetilde{\mathbf{x}},\rho,{\widetilde B}})$ be a generalized seed of an acyclic generalized cluster algebra of geometric type.

\begin{definition}\label{3.1}
For any $i \in [1,n],$  define  a new generalized seed by
$${\Sigma^{(i)}} = ( {{\widetilde{\mathbf{x}}^{(i)}},{\rho ^{(i)}},{{\widetilde B}^{(i)}}} ):= \mu _{{i}}\cdots\mu _{{2}}\mu _{{1}}( {\widetilde{\mathbf{x}},\rho ,{\widetilde B}} ),$$
where ${\widetilde{\mathbf{x}}^{(i)}} = \left\{ {x_1^{(i)},x_2^{(i)}, \ldots ,x_n^{(i)},x_{n + 1}^{(i)} \ldots ,x_m^{(i)}}\right\}.$
The  cluster ${\mathbf{x}}^{(n)} = \left\{ {x_1^{(n)},x_2^{(n)}, \ldots ,x_n^{(n)}} \right\}$ is called  the generalized projective cluster, and each  cluster variable in ${\mathbf{x}}^{(n)}$ is called a generalized projective cluster variable.
\end{definition}

\begin{remark}
\begin{itemize}
\item[(1)]  It follows from Definition \ref{3.1} that  ${B^{(n)}}={B}$, $x_i^{(i)} = x_i^{(j)}$ for any $1\leq i\leq j\leq n$ and $x_i^{(j)}=x_i$ for any $i\in [n+1,m]$, $j\in [1,n]$;
\item[(2)] In fact, the new generalized seed ${\Sigma^{(i)}}$ is obtained by applying a sequence of mutations on $\Sigma$ corresponding to a sink sequence of the directed graph $\Gamma(\Sigma)$.
\end{itemize}
\end{remark}

\begin{definition}\label{3.2}
The lower bound generalized cluster algebra ${\mathcal{L}^{(n)}}( \Sigma)$
is defined to be generated by all generalized projective cluster variables  associated with the generalized seed $\Sigma$ over $\ZZ\mathbb{P}$, i.e.,  ${\mathcal{L}^{(n)}}(\Sigma) = \ZZ\mathbb{P}[{x_1},x_1^{(n)}, \ldots {x_n},x_n^{(n)}]$.
\end{definition}

It is obvious that ${\mathcal{L}^{(n)}}(\Sigma)$ is a finitely generated subalgebra of the generalized cluster algebra $\mathcal{A}(\Sigma)$.

\begin{proposition}\label{prop}
The algebra $\mathbb{ZP}[{x_1},x'_1, \ldots {x_n},x'_n]$ is a $\mathbb{ZP}$-subalgebra of ${\mathcal{L}^{(n)}}(\Sigma).$
\end{proposition}
\begin{proof}
According to the exchange relation, we have
\begin{eqnarray}\label{ex1}	
x'_i &=& x_i^{ - 1}\left(\sum\limits_{r = 0}^{{d_i}} {{\rho _{i,r}}\prod\limits_{j = 1}^{i - 1} {x_j^{({d_i} - r)( - \beta _{ji}^{})}} } \prod\limits_{j = i}^n {x_j^{r\beta _{ji}^{}}} \prod\limits_{j = n + 1}^m {x_j^{\left\lfloor {{{\left[ {r\beta _{ji}^{}} \right]}_ + }} \right\rfloor  + \left\lfloor {{{\left[ {({d_i} - r)( - \beta _{ji}^{})} \right]}_ + }} \right\rfloor }}\right).
\end{eqnarray}
	
Note that the matrices ${B^{\left( {k} \right)}},k \in [1,i-1]$ look as follows:
	
\[{B^{\left( k \right)}} = \left( {\begin{array}{*{20}{c}}
		0&{{b_{1,2}}}& \cdots &{{b_{1,k - 1}}}&{{b_{1,k}}}&{ - {b_{1,k + 1}}}& \cdots &{ - {b_{1,n}}}\\
		{{b_{2,1}}}&0& \cdots &{{b_{2,k - 1}}}&{{b_{2,k}}}&{ - {b_{2,k + 1}}}& \cdots &{ - {b_{2,n}}}\\
		\vdots & \vdots & \ddots &{\vdots}&{\vdots}&{\vdots}&{\ddots}& \vdots \\
		{{b_{k - 1,1}}}&{{b_{k - 1,2}}}& \cdots &0&{{b_{k - 1,k}}}&{ - {b_{k - 1,k + 1}}}& \cdots &{ - {b_{k - 1,n}}}\\
		{{b_{k,1}}}&{{b_{k,2}}}& \cdots &{{b_{k,k - 1}}}&0&{ - {b_{k,k + 1}}}& \cdots &{ - {b_{k,n}}}\\
		{ - {b_{k + 1,1}}}&{ - {b_{k + 1,2}}}& \cdots &{ - {b_{k + 1,k - 1}}}&{ - {b_{k + 1,k}}}&0& \cdots &{{b_{k + 1,n}}}\\
		{ - {b_{k + 2,1}}}&{ - {b_{k + 2,2}}}& \cdots &{ - {b_{k + 2,k - 1}}}&{ - {b_{k + 2,k}}}&{{b_{k + 2,k + 1}}}& \cdots &{{b_{k + 2,n}}}\\
		\vdots & \vdots &{\ddots}&{\vdots}&{\vdots}&{\vdots}& \ddots & \vdots \\
		{ - {b_{n,1}}}&{ - {b_{n,1}}}& \cdots &{ - {b_{n,k - 1}}}&{ - {b_{n,k}}}&{{b_{n,k + 1}}}& \cdots &0
\end{array}} \right).\]
	
According to the exchange relation, we have
\begin{eqnarray*}
	x_k^{(k)}& = &x_k^{ - 1}\Big(\sum\limits_{r = 0}^{{d_k}} {{\rho _{k,r}}} \prod\limits_{j = 1}^m {\left( {x_j^{(k - 1)}} \right)_{}^{\left\lfloor {{{\left[ {r\beta _{jk}^{(k - 1)}} \right]}_ + }} \right\rfloor  + \left\lfloor {{{\left[ {({d_k} - r)( - \beta _{jk}^{(k - 1)})} \right]}_ + }} \right\rfloor }} \Big)\\
	&= &x_k^{ - 1}\Big(\sum\limits_{r = 0}^{{d_k}} {{\rho _{k,r}}\prod\limits_{j = 1}^{k - 1} {{{\left( {x_j^{(k - 1)}} \right)}^{\left\lfloor {{{\left[ {r\beta _{jk}^{(k - 1)}} \right]}_ + }} \right\rfloor  + \left\lfloor {{{\left[ {({d_k} - r)( - \beta _{jk}^{(k - 1)})} \right]}_ + }} \right\rfloor }}} } \\
	&&\cdot \prod\limits_{j = k}^m {x_j^{\left\lfloor {{{\left[ {r\beta _{jk}^{(k - 1)}} \right]}_ + }} \right\rfloor  + \left\lfloor {{{\left[ {({d_k} - r)( - \beta _{jk}^{(k - 1)})} \right]}_ + }} \right\rfloor }} \Big)\\
	&=& x_k^{ - 1}\Big(\sum\limits_{r = 0}^{{d_k}} {{\rho _{k,r}}\prod\limits_{j = k}^n {x_j^{r{\beta _{jk}}}} \prod\limits_{j = 1}^{k - 1} {{{\left( {x_j^{(j)}} \right)}^{ - r{\beta _{jk}}}}} } \prod\limits_{j = n + 1}^m {x_j^{\left\lfloor {{{\left[ {r\beta _{jk}^{(k - 1)}} \right]}_ + }} \right\rfloor  + \left\lfloor {{{\left[ {({d_k} - r)( - \beta _{jk}^{(k - 1)})} \right]}_ + }} \right\rfloor }} \Big)\\
	&=& x_k^{ - 1}\Big(\sum\limits_{r = 0}^{{d_k}} {{\rho _{k,r}}\prod\limits_{j = k}^n {x_j^{r{\beta _{jk}}}} \prod\limits_{j = 1}^{k - 1} {{{\left( {x_j^{(j)}} \right)}^{ - r{\beta _{jk}}}}} } \prod\limits_{\mathop {j{\rm{ }} = {\rm{ }}n{\rm{ }} + {\rm{ }}1}\limits_{b_{jk}^{(k - 1)} \ge 0} }^m {x_j^{\left\lfloor {{{\left[ {r\beta _{jk}^{(k - 1)}} \right]}_ + }} \right\rfloor }} \\
	&&\cdot \prod\limits_{\mathop {j{\rm{ }} = {\rm{ }}n{\rm{ }} + {\rm{ }}1}\limits_{b_{jk}^{(k - 1)} < 0} }^m {x_j^{\left\lfloor {{{\left[ {({d_k} - r)( - \beta _{jk}^{(k - 1)})} \right]}_ + }} \right\rfloor }} \Big)\\
	&=& x_k^{ - 1}\Big(\sum\limits_{r = 1}^{{d_k}} {{\rho _{k,r}}\prod\limits_{j = k}^n {x_j^{r{\beta _{jk}}}} \prod\limits_{j = 1}^{k - 1} {{{\left( {x_j^{(j)}} \right)}^{ - r{\beta _{jk}}}}} } \prod\limits_{\mathop {j{\rm{ }} = {\rm{ }}n{\rm{ }} + {\rm{ }}1}\limits_{b_{jk}^{(k - 1)} \ge 0} }^m {x_j^{\left\lfloor {{{\left[ {r\beta _{jk}^{(k - 1)}} \right]}_ + }} \right\rfloor }} \\
	&&\cdot \prod\limits_{\mathop {j{\rm{ }} = {\rm{ }}n{\rm{ }} + {\rm{ }}1}\limits_{b_{jk}^{(k - 1)} < 0} }^m {x_j^{\left\lfloor {{{\left[ {({d_k} - r)( - \beta _{jk}^{(k - 1)})} \right]}_ + }} \right\rfloor }}  + \prod\limits_{\mathop {j{\rm{ }} = {\rm{ }}n{\rm{ }} + {\rm{ }}1}\limits_{b_{jk}^{(k - 1)} < 0} }^m {x_j^{ - b_{jk}^{(k - 1)}}} \Big).
\end{eqnarray*}	
We rewrite $x_{k}^{(k)}$ as
\begin{equation}\label{(xpk)}
x_k^{(k)} = x_k^{ - 1}\Big({A^{(k)}} + \prod\limits_{\mathop {j{\rm{ }} = {\rm{ }}n{\rm{ }} + {\rm{ }}1}\limits_{b_{jk}^{(k - 1)} < 0} }^m {x_j^{ - b_{jk}^{(k - 1)}}} \Big),
\end{equation}
where
\begin{eqnarray}\label{xin1}
		{A^{(k)}} &= &\sum\limits_{r = 1}^{{d_k}} {{\rho _{k,r}}\Big(\prod\limits_{j = k}^n {x_j^{r{\beta _{jk}}}} \prod\limits_{j = 1}^{k - 1} {{{\left( {x_j^{(j)}} \right)}^{ - r{\beta _{jk}}}}} } \prod\limits_{\mathop {j{\rm{ }} = {\rm{ }}n{\rm{ }} + {\rm{ }}1}\limits_{b_{jk}^{(k - 1)} \ge 0} }^m {x_j^{\left\lfloor {{{\left[ {r\beta _{jk}^{(k - 1)}} \right]}_ + }} \right\rfloor }} \nonumber\\
		&&\cdot \prod\limits_{\mathop {j{\rm{ }} = {\rm{ }}n{\rm{ }} + {\rm{ }}1}\limits_{b_{jk}^{(k - 1)} < 0} }^m {x_j^{\left\lfloor {{{\left[ {({d_k} - r)( - \beta _{jk}^{(k - 1)})} \right]}_ + }} \right\rfloor }} \Big).
\end{eqnarray}
Note that $\beta _{ki}^{(k - 1)} = \beta _{ki}^{} \le 0\  (k \in \left[ {1,i - 1} \right]).$  Thus, for any $ {j \in \left[ {n + 1,m} \right]},$ we have
\begin{eqnarray}\label{(3.6)}	
\beta _{ji}^{(i - 1)} &=& {\beta _{ji}} + \sum\limits_{k = 1}^{i - 1} {\frac{{\left| {b_{jk}^{(k - 1)}} \right|\beta _{ki}^{(k - 1)} + b_{jk}^{(k - 1)}\left| {\beta _{ki}^{(k - 1)}} \right|}}{2}}= {\beta _{ji}} - \sum\limits_{{S}}^{} {  b_{jk}^{(k - 1)}\beta _{ki}},
\end{eqnarray}
where ${S}:=\left\{ {k\left| {b_{jk}^{(k - 1)} < 0,k \in \left[ {1,i - 1} \right]} \right.} \right\}.$

Now we consider ${\Sigma^{(i)}} = {\mu _{i-1}} ( {{\widetilde{\mathbf{x}}^{(i-1)}},{\rho ^{(i-1)}},{{\widetilde B}^{(i-1)}}} )=( {{\widetilde{\mathbf{x}}^{(i)}},{\rho ^{(i)}},{{\widetilde B}^{(i)}}} )$, and  thus
\begin{eqnarray*}
x_i^{(i)} &=&x_i^{ - 1}\Big(\sum\limits_{r = 0}^{{d_i}} {{\rho _{i,r}}\prod\limits_{j = i}^n {x_j^{r{\beta _{ji}}}} \prod\limits_{j = 1}^{i - 1} {{{\left( {x_j^{(j)}} \right)}^{ - r{\beta _{ji}}}}} \prod\limits_{j = n + 1}^m {x_j^{\left\lfloor {{{\left[ {r\beta _{ji}^{(i - 1)}} \right]}_ + }} \right\rfloor  + \left\lfloor {{{\left[ {({d_i} - r)( - \beta _{ji}^{(i - 1)})} \right]}_ + }} \right\rfloor }} } \Big).
\end{eqnarray*}

We have the following two cases.
\begin{itemize}[leftmargin=*]
\item[(1)]  When $b_{ki}=0$ for any $k \in [1,i-1]$, we obtain
\[x_i^{(i)} = x_i^{ - 1}\Big(\sum\limits_{r = 0}^{{d_i}} {{\rho _{i,r}}\prod\limits_{j = i}^n {x_j^{r{\beta _{ji}}}} \prod\limits_{j = n + 1}^m {x_j^{\left\lfloor {{{\left[ {r\beta _{ji}^{(i - 1)}} \right]}_ + }} \right\rfloor  + \left\lfloor {{{\left[ {({d_i} - r)( - \beta _{ji}^{(i - 1)})} \right]}_ + }} \right\rfloor }} } \Big),\]
and  Equation~(\ref{(3.6)}) becomes  $\beta _{ji}^{(i - 1)} = {\beta _{ji}}.$
Thus, by combining Equation~(\ref{ex1}), we have
\[x_i^{(i)} = x_i^{ - 1}\Big(\sum\limits_{r = 0}^{{d_i}} {{\rho _{i,r}}\prod\limits_{j = i}^n {x_j^{r{\beta _{ji}}}} \prod\limits_{j = n + 1}^m {x_j^{\left\lfloor {{{\left[ {r\beta _{ji}^{}} \right]}_ + }} \right\rfloor  + \left\lfloor {{{\left[ {({d_i} - r)( - \beta _{ji}^{})} \right]}_ + }} \right\rfloor }} } \Big)= {x'_i}.\]
Therefore $${x'_i} \in \mathbb{ZP}[{x_1},x_1^{(n)}, \ldots {x_n},x_n^{(n)}].$$

\item[(2)] When $b_{ki}\ne0$  for some $k \in [1,i - 1]$. According to  Equation~(\ref{(xpk)}), we  have
\begin{eqnarray*}
	x_i^{(i)} &=& x_i^{ - 1}\Big(\sum\limits_{r = 0}^{{d_i}} {{\rho _{i,r}}\prod\limits_{j = i}^n {x_j^{r{\beta _{ji}}}} \prod\limits_{k = 1}^{i - 1} {{{(x_{k}^{ - 1}{\rm{(}}{A^{(k)}} + \prod\limits_{\mathop {j = n + 1}\limits_{b_{jk }^{(k - 1)} < 0} }^m {x_j^{ - b_{jk}^{(k - 1)}}} ))}^{- r{\beta _{ki}}}}} } \\
	&&\cdot\prod\limits_{j = n + 1}^m {x_j^{\left\lfloor {{{\left[ {r\beta _{ji}^{(i - 1)}} \right]}_ + }} \right\rfloor  + \left\lfloor {{{\left[ {({d_i} - r)( - \beta _{ji}^{(i - 1)})} \right]}_ + }} \right\rfloor }} \Big)\\
	&= &x_i^{ - 1}\Big(\sum\limits_{r = 0}^{{d_i}} {{\rho _{i,r}}\prod\limits_{j = i}^n {x_j^{r{\beta _{ji}}}} \prod\limits_{k = 1}^{i - 1} {x_k^{r{\beta _{ki}}}} \prod\limits_{k = 1}^{i - 1} {{{{\rm{(}}{A^{(k)}} + \prod\limits_{\mathop {j = n + 1}\limits_{b_{jk}^{(k - 1)} < 0} }^m {x_j^{ - b_{jk}^{(k - 1)}}} )}^{ - r{\beta _{ki}}}}} } \\
	&&\cdot\prod\limits_{j = n + 1}^m {x_j^{\left\lfloor {{{\left[ {r\beta _{ji}^{(i - 1)}} \right]}_ + }} \right\rfloor  + \left\lfloor {{{\left[ {({d_i} - r)( - \beta _{ji}^{(i - 1)})} \right]}_ + }} \right\rfloor }} \Big).
\end{eqnarray*}

By multiplying both sides  by {\small$\prod\limits_{k = 1}^{i - 1} {x_k^{ - {b_{ki}}}} $}, we have
\begin{eqnarray*}
	\prod\limits_{k = 1}^{i - 1} {x_k^{ - {b_{ki}}}} x_i^{(i)} &=& x_i^{ - 1}\Big(\sum\limits_{r = 0}^{{d_i}} {{\rho _{i,r}}\prod\limits_{j = i}^n {x_j^{r{\beta _{ji}}}} \prod\limits_{k = 1}^{i - 1} {x_k^{ - {b_{ki}} + r{\beta _{ki}}}} \prod\limits_{k = 1}^{i - 1} {{{{\rm{(}}{A^{(k)}} + \prod\limits_{\mathop {j = n + 1}\limits_{b_{jk}^{(k - 1)} < 0} }^m {x_j^{ - b_{jk}^{(k - 1)}}} )}^{ - r{\beta _{ki}}}}} } \\
	&&\cdot\prod\limits_{j = n + 1}^m {x_j^{\left\lfloor {{{\left[ {r\beta _{ji}^{(i - 1)}} \right]}_ + }} \right\rfloor  + \left\lfloor {{{\left[ {({d_i} - r)( - \beta _{ji}^{(i - 1)})} \right]}_ + }} \right\rfloor }} \Big)\\
	&= &x_i^{ - 1}\Big(\sum\limits_{r = 0}^{{d_i}} {{\rho _{i,r}}\prod\limits_{j = i}^n {x_j^{r{\beta _{ji}}}} \prod\limits_{j = 1}^{i - 1} {x_j^{(r - {d_i}){\beta _{ji}}}} \prod\limits_{k = 1}^{i - 1} {\prod\limits_{\mathop {j = n + 1}\limits_{b_{jk}^{(k - 1)} < 0} }^m {x_j^{r{\beta _{ki}}b_{jk}^{(k - 1)}}} } } \\
	&&\cdot\prod\limits_{j = n + 1}^m {x_j^{\left\lfloor {{{\left[ {r\beta _{ji}^{(i - 1)}} \right]}_ + }} \right\rfloor  + \left\lfloor {{{\left[ {({d_i} - r)( - \beta _{ji}^{(i - 1)})} \right]}_ + }} \right\rfloor }} \Big) + f\\
		&=& x_i^{ - 1}\Big(\sum\limits_{r = 0}^{{d_i}} {{\rho _{i,r}}\prod\limits_{j = i}^n {x_j^{r{\beta _{ji}}}} \prod\limits_{j = 1}^{i - 1} {x_j^{(r - {d_i}){\beta _{ji}}}} \prod\limits_{j = n + 1}^m {\prod\limits_{\mathop {k = 1}\limits_{b_{jk}^{(k - 1)} < 0} }^{i - 1} {x_j^{r{\beta _{ki}}b_{jk}^{(k - 1)}}} } } \\
		&&\cdot \prod\limits_{j = n + 1}^m {x_j^{\left\lfloor {{{\left[ {r\beta _{ji}^{(i - 1)}} \right]}_ + }} \right\rfloor  + \left\lfloor {{{\left[ {({d_i} - r)( - \beta _{ji}^{(i - 1)})} \right]}_ + }} \right\rfloor }} \Big) + f,
\end{eqnarray*}
where
\begin{eqnarray}\label{xin2}
		f &=& x_i^{ - 1}\Big(\sum\limits_{r = 0}^{{d_i}} {{\rho _{i,r}}\prod\limits_{j = i}^n {x_j^{r{\beta _{ji}}}} \prod\limits_{j = 1}^{i - 1} {x_j^{(r - {d_i}){\beta _{ji}}}} \prod\limits_{j = n + 1}^m {x_j^{\left\lfloor {{{\left[ {r\beta _{ji}^{(i - 1)}} \right]}_ + }} \right\rfloor  + \left\lfloor {{{\left[ {({d_i} - r)( - \beta _{ji}^{(i - 1)})} \right]}_ + }} \right\rfloor }} }\nonumber \\
	&& \cdot \prod\limits_{k = 1}^{i - 1} {\sum\limits_{t = 1}^{ - r{\beta _{ki}}} {(C_{ - r{\beta _{ki}}}^t{{({A^{(k)}})}^t}\prod\limits_{\mathop {j{\rm{ }} = {\rm{ }}n{\rm{ }} + {\rm{ }}1}\limits_{b_{{j_k}}^{(k - 1)} < 0} }^m {x_j^{b_{{j_k}}^{(k - 1)}(r{\beta _{ki}} + t)}} } } )\Big),
\end{eqnarray}
with $C_s^l: = \left( {\begin{array}{*{20}{c}}
		s\\
		l
\end{array}} \right).$

We first claim that
\begin{eqnarray}\label{include}
f \in \mathbb{ZP}\left[ {{x_1}, \ldots ,{x_n},x_1^{(1)}, \ldots ,x_{i - 1}^{(i - 1)}} \right].
\end{eqnarray}
Note that we can rewrite Equation~(\ref{xin1}) as follows:
\begin{eqnarray}\label{xin3}
	{A^{(k)}} = \sum\limits_{r = 1}^{{d_k}} {{P_{k,r}}\prod\limits_{j = k}^n {x_j^{r{\beta _{jk}}}} \prod\limits_{j = 1}^{k - 1} {{{\left( {x_j^{(j)}} \right)}^{ - r{\beta _{jk}}}}} },
\end{eqnarray}
where ${P_{k,r}} = {\rho _{k,r}}\prod\limits_{\mathop {j{\rm{ }} = {\rm{ }}n{\rm{ }} + {\rm{ }}1}\limits_{b_{jk}^{(k - 1)} \ge 0} }^m {x_j^{\left\lfloor {{{\left[ {r\beta _{jk}^{(k - 1)}} \right]}_ + }} \right\rfloor }} \prod\limits_{\mathop {j{\rm{ }} = {\rm{ }}n{\rm{ }} + {\rm{ }}1}\limits_{b_{jk}^{(k - 1)} < 0} }^m {x_j^{\left\lfloor {{{\left[ {({d_k} - r)( - \beta _{jk}^{(k - 1)})} \right]}_ + }} \right\rfloor }}.$

Thus we have
\[x_i^{ - 1}{A^{(k)}} = \sum\limits_{r = 1}^{{d_k}} {{P_{k,r}}x_i^{r{\beta _{ik}} - 1}\prod\limits_{j = k,j \ne i}^n {x_j^{r{\beta _{jk}}}} \prod\limits_{j = 1}^{k - 1} {{{\left( {x_j^{(j)}} \right)}^{ - r{\beta _{jk}}}}} }. \]

Note that, for any $r\in[1,d_k]$, we have $r{\beta _{i,k}} -1 \ge 0.$ So we obtain $x_i^{ - 1}{A^{(k)}} \in \mathbb{ZP}\left[ {x_1^{(1)}, \ldots ,x_{k - 1}^{(k - 1)}},x_{k+1},\ldots ,{x_n} \right]$.

Then, by Equation~(\ref{xin2}), we get $f \in \mathbb{ZP}\left[ {{x_1}, \ldots ,{x_n},x_1^{(1)}, \ldots ,x_{i - 1}^{(i - 1)}} \right]$.

Now we continue to compute
\begin{eqnarray*}
\prod\limits_{k = 1}^{i - 1} {x_k^{ - {b_{ki}}}} x_i^{(i)}	&=& x_i^{ - 1}\Big(\sum\limits_{r = 0}^{{d_i}} {{\rho _{i,r}}\prod\limits_{j = i}^n {x_j^{r{\beta _{ji}}}} \prod\limits_{j = 1}^{i - 1} {x_j^{(r - {d_i}){\beta _{ji}}}} \prod\limits_{j = n + 1}^m {\prod\limits_S^{} {x_j^{r{\beta _{ki}}b_{jk}^{(k - 1)}}} } } \prod\limits_{\mathop {j{\rm{ }} = {\rm{ }}n{\rm{ }} + {\rm{ }}1}\limits_{b_{ji}^{(i - 1)} \ge 0} }^m {x_j^{\left\lfloor {{{\left[ {r\beta _{ji}^{(i - 1)}} \right]}_ + }} \right\rfloor }} \\
	&&\cdot \prod\limits_{\mathop {j{\rm{ }} = {\rm{ }}n{\rm{ }} + {\rm{ }}1}\limits_{b_{ji}^{(i - 1)} < 0,{b_{ji}} \ge 0} }^m {x_j^{\left\lfloor {{{\left[ {({d_i} - r)( - \beta _{ji}^{(i - 1)})} \right]}_ + }} \right\rfloor }} \prod\limits_{\mathop {j{\rm{ }} = {\rm{ }}n{\rm{ }} + {\rm{ }}1}\limits_{b_{ji}^{(i - 1)} < 0,{b_{ji}} < 0} }^m {x_j^{\left\lfloor {{{\left[ {({d_i} - r)( - \beta _{ji}^{(i - 1)})} \right]}_ + }} \right\rfloor }} \Big) + f\\
	&=& x_i^{ - 1}\Big(\sum\limits_{r = 0}^{{d_i}} {{\rho _{i,r}}\prod\limits_{j = i}^n {x_j^{r{\beta _{ji}}}} \prod\limits_{j = 1}^{i - 1} {x_j^{(r - {d_i}){\beta _{ji}}}} \prod\limits_{\mathop {j{\rm{ }} = {\rm{ }}n{\rm{ }} + {\rm{ }}1}\limits_{b_{ji}^{(i - 1)} \ge 0} }^m {\prod\limits_S^{} {x_j^{r{\beta _{ki}}b_{jk}^{(k - 1)}}} } } \\
	&&\cdot \prod\limits_{\mathop {j{\rm{ }} = {\rm{ }}n{\rm{ }} + {\rm{ }}1}\limits_{b_{ji}^{(i - 1)} < 0,{b_{ji}} \ge 0} }^m {\prod\limits_S^{} {x_j^{r{\beta _{ki}}b_{jk}^{(k - 1)}}} } \prod\limits_{\mathop {j{\rm{ }} = {\rm{ }}n{\rm{ }} + {\rm{ }}1}\limits_{b_{ji}^{(i - 1)} < 0,{b_{ji}} < 0} }^m {\prod\limits_S^{} {x_j^{r{\beta _{ki}}b_{jk}^{(k - 1)}}} } \prod\limits_{\mathop {j{\rm{ }} = {\rm{ }}n{\rm{ }} + {\rm{ }}1}\limits_{b_{ji}^{(i - 1)} \ge 0} }^m {x_j^{\left\lfloor {{{\left[ {r\beta _{ji}^{(i - 1)}} \right]}_ + }} \right\rfloor }} \\
	&&\cdot \prod\limits_{\mathop {j{\rm{ }} = {\rm{ }}n{\rm{ }} + {\rm{ }}1}\limits_{b_{ji}^{(i - 1)} < 0,{b_{ji}} \ge 0} }^m {x_j^{\left\lfloor {{{\left[ {({d_i} - r)( - \beta _{ji}^{(i - 1)})} \right]}_ + }} \right\rfloor }} \prod\limits_{\mathop {j{\rm{ }} = {\rm{ }}n{\rm{ }} + {\rm{ }}1}\limits_{b_{ji}^{(i - 1)} < 0,{b_{ji}} < 0} }^m {x_j^{\left\lfloor {{{\left[ {({d_i} - r)( - \beta _{ji}^{(i - 1)})} \right]}_ + }} \right\rfloor }} \Big) + f.
\end{eqnarray*}

We rewrite the above equation as follows:
\begin{eqnarray}\label{(3.8)}
\prod\limits_{k = 1}^{i - 1} {x_k^{ - {b_{ki}}}} x_i^{(i)}	&= &x_i^{ - 1}\Big(\sum\limits_{r = 0}^{{d_i}} {{\rho _{i,r}}\prod\limits_{j = i}^n {x_j^{r{\beta _{ji}}}} \prod\limits_{j = 1}^{i - 1} {x_j^{(r - {d_i}){\beta _{ji}}}} \prod\limits_{\mathop {j{\rm{ }} = {\rm{ }}n{\rm{ }} + {\rm{ }}1}\limits_{b_{ji}^{(i - 1)} \ge 0} }^m {\prod\limits_S^{} {x_j^{r{\beta _{ki}}b_{jk}^{(k - 1)}}} } } \prod\limits_{\mathop {j{\rm{ }} = {\rm{ }}n{\rm{ }} + {\rm{ }}1}\limits_{b_{ji}^{(i - 1)} \ge 0} }^m {x_j^{\left\lfloor {{{\left[ {r\beta _{ji}^{(i - 1)}} \right]}_ + }} \right\rfloor }}\nonumber \\
	&&\cdot \prod\limits_{\mathop {j{\rm{ }} = {\rm{ }}n{\rm{ }} + {\rm{ }}1}\limits_{b_{ji}^{(i - 1)} < 0,{b_{ji}} \ge 0} }^m {\prod\limits_S^{} {x_j^{r{\beta _{ki}}b_{jk}^{(k - 1)}}} } \prod\limits_{\mathop {j{\rm{ }} = {\rm{ }}n{\rm{ }} + {\rm{ }}1}\limits_{b_{ji}^{(i - 1)} < 0,{b_{ji}} \ge 0} }^m {x_j^{\left\lfloor {{{\left[ {({d_i} - r)( - \beta _{ji}^{(i - 1)})} \right]}_ + }} \right\rfloor }} \nonumber\\
	&&\cdot \prod\limits_{\mathop {j{\rm{ }} = {\rm{ }}n{\rm{ }} + {\rm{ }}1}\limits_{b_{ji}^{(i - 1)} < 0,{b_{ji}} < 0} }^m {\prod\limits_S^{} {x_j^{r{\beta _{ki}}b_{jk}^{(k - 1)}}} } \prod\limits_{\mathop {j{\rm{ }} = {\rm{ }}n{\rm{ }} + {\rm{ }}1}\limits_{b_{ji}^{(i - 1)} < 0,{b_{ji}} < 0} }^m {x_j^{\left\lfloor {{{\left[ {({d_i} - r)( - \beta _{ji}^{(i - 1)})} \right]}_ + }} \right\rfloor }} \Big) + f.
\end{eqnarray}

For each $ j \in[n+1,m] $, we need the following discussions.
\begin{itemize}[leftmargin=*]
\item[(a)] If $b _{ji}^{(i - 1)} \ge 0$, we have
\begin{eqnarray*}
	& &x_j^{\left\lfloor {r\beta _{ji}^{(i - 1)}} \right\rfloor }\prod\limits_S^{} {x_j^{r{\beta _{ki}}b_{jk}^{(k - 1)}}} \mathop = \limits^{(\ref{(3.6)})} x_j^{\left\lfloor {r{\beta _{ji}} - \sum\limits_S^{} {rb_{jk}^{(k - 1)}\beta _{ki}^{(k - 1)}} } \right\rfloor }\prod\limits_S^{} {x_j^{r{\beta _{ki}}b_{jk}^{(k - 1)}}} \\
	&= &x_j^{\left\lfloor {r{\beta _{ji}} - \sum\limits_S^{} {rb_{jk}^{(k - 1)}\beta _{ki}^{}} } \right\rfloor  + \sum\limits_S^{} {r{\beta _{ki}}b_{jk}^{(k - 1)}} }= x_j^{\left\lfloor {r{\beta _{ji}}} \right\rfloor }.
\end{eqnarray*}

According to Equation~(\ref{(3.6)}), if $\beta _{ji}^{(i - 1)} \ge 0,$ then  ${\beta _{ji}} = \beta _{ji}^{(i - 1)} + \sum\limits_S {b_{jk}^{(k - 1)}} {\beta _{ki}} \ge 0,$
which implies that if $b_{ji}^{(i - 1)} \ge 0,$ then  ${b _{ji}} \ge 0$.
Therefore, we have
$$\prod\limits_{\mathop {j = n + 1}\limits_{b_{ji}^{(i - 1)} \ge 0} }^m {\prod\limits_S^{} {x_j^{r{\beta _{ki}}b_{jk}^{(k - 1)}}} } \prod\limits_{\mathop {j = n + 1}\limits_{b_{ji}^{(i - 1)} \ge 0} }^m {x_j^{\left\lfloor {{{\left[ {r\beta _{ji}^{(i - 1)}} \right]}_ + }} \right\rfloor }} = \prod\limits_{\mathop {j = n + 1}\limits_{b_{ji}^{(i - 1)} \ge 0} }^m {x_j^{\left\lfloor {r{\beta _{ji}}} \right\rfloor }}  = \prod\limits_{\mathop {j = n + 1}\limits_{b_{ji}^{(i - 1)} \ge 0,{b_{ji}} \ge 0} }^m {x_j^{\left\lfloor {r{\beta _{ji}}} \right\rfloor }} .$$

\item[(b)]  If   $b _{ji}^{(i - 1)} < 0$ and ${b _{ji}} \ge 0$, we have
\begin{eqnarray*}
	&&	x_j^{\left\lfloor {({d_i} - r)( - \beta _{ji}^{(i - 1)})} \right\rfloor }\prod\limits_{{S}}^{} {x_j^{r{\beta _{ki}}b_{jk}^{(k - 1)}}} \mathop = \limits^{(\ref{(3.6)})}  x_j^{\left\lfloor {({d_i} - r)( - {\beta _{ji}} + \sum\limits_{{S}}^{} {b_{jk}^{(k - 1)}\beta _{ki}^{(k - 1)}} )} \right\rfloor }\prod\limits_{{S}}^{} {x_j^{r{\beta _{ki}}b_{jk}^{(k - 1)}}} \\
		&=& x_j^{\left\lfloor {({d_i} - r)( - {\beta _{ji}} + \sum\limits_{{S}}^{} {b_{jk}^{(k - 1)}\beta _{ki}^{(k - 1)}} )} \right\rfloor  + \sum\limits_{{S}}^{} {r\beta _{ki}^{(k - 1)}b_{jk}^{(k - 1)}} }
		= x_j^{\left\lfloor {r{\beta _{ji}}} \right\rfloor  - {b_{ji}} + \sum\limits_{{S}}^{} {b _{ki}^{(k - 1)}b_{jk}^{(k - 1)}} }.
\end{eqnarray*}
Thus we obtain
\begin{eqnarray*}
	&&\prod\limits_{\mathop {j{\rm{ }} = {\rm{ }}n{\rm{ }} + {\rm{ }}1}\limits_{b_{ji}^{(i - 1)} < 0,{b_{ji}} \ge 0} }^m {\prod\limits_S^{} {x_j^{r{\beta _{ki}}b_{jk}^{(k - 1)}}} } \prod\limits_{\mathop {j{\rm{ }} = {\rm{ }}n{\rm{ }} + {\rm{ }}1}\limits_{b_{ji}^{(i - 1)} < 0,{b_{ji}} \ge 0} }^m {x_j^{\left\lfloor {{{\left[ {({d_i} - r)( - \beta _{ji}^{(i - 1)})} \right]}_ + }} \right\rfloor }}  \\
	&=& \prod\limits_{\mathop {j = n + 1}\limits_{b _{ji}^{(i - 1)} < 0,{b _{ji}} \ge 0} }^m {x_j^{\left\lfloor {r{\beta _{ji}}} \right\rfloor  - {b_{ji}} + \sum\limits_{{S}}^{} {b _{ki}^{(k - 1)}b_{jk}^{(k - 1)}} }}
 = {\rm{ }}\prod\limits_{\mathop {j{\rm{ }} = {\rm{ }}n{\rm{ }} + {\rm{ }}1}\limits_{b_{ji}^{(i - 1)} < 0,{b_{ji}} \ge 0} }^m ({x_j^{\left\lfloor {r{\beta _{ji}}} \right\rfloor  - {b_{ji}}}} \prod\limits_{\mathop {k{\rm{ }} = {\rm{ }}1}\limits_{b_{jk}^{(k - 1)} < 0} }^{i - 1} {x_j^{b_{ki}^{(k - 1)}b_{jk}^{(k - 1)}}}) .
\end{eqnarray*}

\item[(c)]  If   $b _{ji}^{(i - 1)} < 0$ and ${b _{ji}} < 0$, we have
\begin{eqnarray*}
& &	x_j^{\left\lfloor {({d_i} - r)( - \beta _{ji}^{(i - 1)})} \right\rfloor }\prod\limits_{{S}}^{} {x_j^{r{\beta _{ki}}b_{jk}^{(k - 1)}}}  = x_j^{\left\lfloor {({d_i} - r)( - {\beta _{ji}} + \sum\limits_{{S}}^{} {b_{jk}^{(k - 1)}\beta _{ki}^{(k - 1)}} )} \right\rfloor }\prod\limits_{{S}}^{} {x_j^{r{\beta _{ki}}b_{jk}^{(k - 1)}}} \\
	&=&x_j^{\left\lfloor {({d_i} - r)( - {\beta _{ji}} + \sum\limits_{{S}}^{} {b_{jk}^{(k - 1)}\beta _{ki}^{(k - 1)}} )} \right\rfloor  + \sum\limits_{{S}}^{} {r\beta _{ki}^{(k - 1)}b_{jk}^{(k - 1)}} }
	= x_j^{\left\lfloor {({d_i} - r)( - {\beta _{ji}})} \right\rfloor  + \sum\limits_{{S}}^{} {b_{jk}^{(k - 1)}b_{ki}^{(k - 1)}} }.
\end{eqnarray*}
Thus we obtain
\begin{eqnarray*}
&&	\prod\limits_{\mathop {j{\rm{ }} = {\rm{ }}n{\rm{ }} + {\rm{ }}1}\limits_{b_{ji}^{(i - 1)} < 0,{b_{ji}} < 0} }^m {\prod\limits_S^S {x_j^{r{\beta _{ki}}b_{jk}^{(k - 1)}}} } \prod\limits_{\mathop {j{\rm{ }} = {\rm{ }}n{\rm{ }} + {\rm{ }}1}\limits_{b_{ji}^{(i - 1)} < 0,{b_{ji}} < 0} }^m {x_j^{\left\lfloor {{{\left[ {({d_i} - r)( - \beta _{ji}^{(i - 1)})} \right]}_ + }} \right\rfloor }} \\
	&= &\prod\limits_{\mathop {j = n + 1}\limits_{b _{ji}^{(i - 1)} < 0,{b_{ji}} < 0} }^m {x_j^{\left\lfloor {({d_i} - r)( - {\beta _{ji}})} \right\rfloor  + \sum\limits_{{S}}^{} {b_{jk}^{(k - 1)}b_{ki}^{(k - 1)}} }}
	= \prod\limits_{\mathop {j = n + 1}\limits_{b _{ji}^{(i - 1)} < 0,{b_{ji}} < 0} }^m \big({x_j^{\left\lfloor {({d_i} - r) ( - {\beta _{ji}})} \right\rfloor }} \prod\limits_{\mathop {k = 1}\limits_{b_{jk}^{(k - 1)} < 0} }^{i - 1} {x_j^{b_{jk}^{(k - 1)}b_{ki}^{(k - 1)}}}\big) .
\end{eqnarray*}
\end{itemize}

By substituting the results of (a), (b) and (c) into Equation~(\ref{(3.8)}), we obtain
\begin{eqnarray*}
	\prod\limits_{k = 1}^{i - 1} {x_k^{ - {b_{ki}}}} x_i^{(i)}& = &x_i^{ - 1}\Big(\sum\limits_{r = 0}^{{d_i}} {{\rho _{i,r}}\prod\limits_{j = i}^n {x_j^{r{\beta _{ji}}}} \prod\limits_{j = 1}^{i - 1} {x_j^{(r - {d_i}){\beta _{ji}}}} } \\
    && \cdot \prod\limits_{\mathop {j{\rm{ }} = {\rm{ }}n{\rm{ }} + {\rm{ }}1}\limits_{b_{ji}^{(i - 1)} \ge 0,{b_{ji}} \ge 0} }^m {x_j^{\left\lfloor {r{\beta _{ji}}} \right\rfloor }} \prod\limits_{\mathop {j{\rm{ }} = {\rm{ }}n{\rm{ }} + {\rm{ }}1}\limits_{b_{ji}^{(i - 1)} < 0,{b_{ji}} \ge 0} }^m {(x_j^{\left\lfloor {r{\beta _{ji}}} \right\rfloor  - {b_{ji}}}\prod\limits_{\mathop {k{\rm{ }} = {\rm{ }}1}\limits_{b_{jk}^{(k - 1)} < 0} }^{i - 1} {x_j^{b_{ki}^{(k - 1)}b_{jk}^{(k - 1)}})} } \\
	&&\cdot \prod\limits_{\mathop {j = n + 1}\limits_{b _{ji}^{(i - 1)} < 0,{b_{ji}} < 0} }^m {(x_j^{\left\lfloor {({d_i} - r)( - {\beta _{ji}})} \right\rfloor }\prod\limits_{\mathop {k = 1}\limits_{b_{jk}^{(k - 1)} < 0} }^{i - 1} {x_j^{b_{jk}^{(k - 1)}b_{ki}^{(k - 1)}}} )} \Big) + f\\
	&	= &x_i^{ - 1}\Big(\sum\limits_{r = 0}^{{d_i}} {{\rho _{i,r}}\prod\limits_{j = i}^n {x_j^{r{\beta _{ji}}}} \prod\limits_{j = 1}^{i - 1} {x_j^{(r - {d_i}){\beta _{ji}}}} } \\
	&&	\cdot \prod\limits_{\mathop {j = n + 1}\limits_{b _{ji}^{(i - 1)} \ge 0,{b _{ji}} \ge 0} }^m {x_j^{\left\lfloor {r{\beta _{ji}}} \right\rfloor }} \prod\limits_{\mathop {j = n + 1}\limits_{b_{ji}^{(i - 1)} < 0,{b_{ji}} \ge 0} }^m {x_j^{\left\lfloor {r{\beta _{ji}}} \right\rfloor }} \prod\limits_{\mathop {j = n + 1}\limits_{b _{ji}^{(i - 1)} < 0,{b _{ji}} < 0} }^m {x_j^{\left\lfloor {({d_i} - r)( - {\beta _{ji}})} \right\rfloor }} )\\
    && \cdot \prod\limits_{\mathop {j{\rm{ }} = {\rm{ }}n{\rm{ }} + {\rm{ }}1}\limits_{b_{ji}^{(i - 1)} < 0} }^m ({x_j^{ - {{\left[ {{b_{ji}}} \right]}_ + }}\prod\limits_{\mathop {k{\rm{ }} = {\rm{ }}1}\limits_{b_{jk}^{(k - 1)} < 0} }^{i - 1} {x_j^{b_{ki}^{(k - 1)}b_{jk}^{(k - 1)}}} }\Big)  + f.
\end{eqnarray*}

Note that, according to Equation~(\ref{(3.6)}), for any $j \in [n+1,m]$, if $b_{ji} <0$, then $\beta _{ji}^{(i - 1)} < 0$. Thus we have
\[\prod\limits_{\mathop {j = n + 1}\limits_{b_{ji}^{(i - 1)} \ge 0,{b_{ji}} < 0} }^m {x_j^{\left\lfloor {({d_i} - r)( - {\beta _{ji}})} \right\rfloor }}  = 1.\]

Hence, we obtain
   \begin{eqnarray*}
   	\prod\limits_{k = 1}^{i - 1} {x_k^{ - {b_{ki}}}} x_i^{(i)} &=&x_i^{ - 1}(\sum\limits_{r = 0}^{{d_i}} {{\rho _{i,r}}\prod\limits_{j = i}^n {x_j^{r{\beta _{ji}}}} \prod\limits_{j = 1}^{i - 1} {x_j^{(r - {d_i}){\beta _{ji}}}} } \prod\limits_{\mathop {j{\rm{ }} = {\rm{ }}n{\rm{ }} + {\rm{ }}1}\limits_{{b_{ji}} \ge 0} }^m {x_j^{\left\lfloor {r{\beta _{ji}}} \right\rfloor }} \prod\limits_{\mathop {j{\rm{ }} = {\rm{ }}n{\rm{ }} + {\rm{ }}1}\limits_{{b_{ji}} < 0} }^m {x_j^{\left\lfloor {({d_i} - r)( - {\beta _{ji}})} \right\rfloor }} )\\
&&\cdot\prod\limits_{\mathop {j = n + 1}\limits_{b_{ji}^{(i - 1)} < 0} }^m ({x_j^{ - {{\left[ {{b_{ji}}} \right]}_ + }}\prod\limits_{\mathop {k = 1}\limits_{b_{jk}^{(k - 1)} < 0} }^{i - 1} {x_j^{b_{ki}^{(k - 1)}b_{jk}^{(k - 1)}}} })  + f\\
&\mathop = \limits^{(\ref{ex1})}& {x'_i}\prod\limits_{\mathop {j{\rm{ }} = {\rm{ }}n{\rm{ }} + {\rm{ }}1}\limits_{b_{ji}^{(i - 1)} < 0} }^m( {x_j^{ - {{\left[ {{b_{ji}}} \right]}_ + }}\prod\limits_{\mathop {k{\rm{ }} = {\rm{ }}1}\limits_{b_{jk}^{(k - 1)} < 0} }^{i - 1} {x_j^{b_{ki}^{(k - 1)}b_{jk}^{(k - 1)}}} })  + f.
\end{eqnarray*}
Then by combining Equation~(\ref{include}), we get
\[x'_i \in \mathbb{ZP}[{x_1},x_1^{(n)}, \ldots ,{x_n},x_n^{(n)}].\]
\end{itemize}
The proof is completed.
\end{proof}

\begin{theorem}\label{thm1}
Let  ${\Sigma}=( {\widetilde{\mathbf{x}},\rho,{\widetilde B}})$ be an acyclic and coprime generalized seed. Then we have $\mathcal{A}( {\Sigma})={\mathcal{L}^{(n)}}( \Sigma).$
\end{theorem}

\begin{proof}
It follows from Theorem \ref{standard-1} and Proposition \ref{prop}.
\end{proof}

 \begin{definition}
 A monomial in ${x_1},x_1^{(n)}, \ldots ,{x_n},x_n^{(n)}$ is called  a projective standard  monomial if it contains no
 product of the form ${x_i}x_i^{(n)}$ for any $i \in [1,n]$.
 \end{definition}

\begin{remark}\label{rem}
For any $i \in [1,n],$ we have $$x_i^{\left( n \right)}{x_i} = \sum\limits_{r = 0}^{{d_i}} {{\rho _{i,r}}\prod\limits_{j = i}^n {x_j^{r{\beta _{ji}}}} \prod\limits_{j = 1}^{i - 1} {{{\left( {x_j^{(i - 1)}} \right)}^{ - r{\beta _{ji}}}}} \prod\limits_{j= n + 1}^m {x_j^{\left\lfloor {{{\left[ {r\beta _{ji}^{\left( {i - 1} \right)}} \right]}_ + }} \right\rfloor  + \left\lfloor {{{\left[ {\left( {{d_i} - r} \right)\left( { - \beta _{ji}^{\left( {i - 1} \right)}} \right)} \right]}_ + }} \right\rfloor }} } .$$  Therefore the projective standard monomials in ${x_1},x_1^{(n)}, \ldots ,{x_n},x_n^{(n)}$ span   ${\mathcal{L}^{(n)}}(\Sigma)$ as a $\mathbb{ZP}$-module.
\end{remark}

For any $\mathbf{a} = \left( {{a_1},{a_2}, \ldots, {a_n}} \right) \in {{\mathbb{Z}}^n},$ we denote
$${\mathbf{x}^\mathbf{a}} := x_1^{{a_1}}x_2^{{a_2}} \ldots x_n^{{a_n}}.$$

Let $\prec$ denote the lexicographic order on $\mathbb{Z}^{n}$, i.e., for any two vectors $\mathbf{a} = \left( {{a_1},{a_2}, \ldots, {a_n}} \right)$, $\mathbf{a}' = \left( {{a'_1},{a'_2}, \ldots, {a'_n}} \right)$ $\in\ZZ^n$ satisfy that $\mathbf{a}\prec\mathbf{a}'$ if and only if  there exists $k \in [1,n]$ such that $a_k<a'_k$ and $a_i=a'_i$ for all $i \in [1,k-1]$.
This order induces the lexicographic order on the Laurent monomials as
\begin{gather*}
	\mathbf{x}^{\mathbf{a}}\prec\mathbf{x}^{\mathbf{a}'}
	\quad \text{ if } \quad
	\mathbf{a}\prec\mathbf{a}'.
\end{gather*}

\begin{definition}	
We call ${\mathbf{x}^\mathbf{a}}$  the first Laurent monomial of the Laurent polynomial $Y = {\mathbf{x}^\mathbf{a}}+\sum\limits_y^{} {{\mathbf{x}^{\mathbf{a}_y}}}$, if $\mathbf{a}_y\prec\mathbf{a}$ for any $y$.	
\end{definition}

\begin{lemma}\label{L3.5}
The first  Laurent monomial of 	$x_i^{(n)}$ is $x_i^{ - 1}\prod\limits_{j = i + 1}^n {x_j^{{k_j}}} P_i,$ where ${k_j} \in {\mathbb{Z}_{ \ge 0}}$ and $0\neq{P_i} \in \mathbb{ZP}.$
\end{lemma}
\begin{proof}
When $k=1$, we have
$$x_1^{(n)}{\rm{ }} =x_1^{ - 1}\Big(\sum\limits_{r = 0}^{{d_1}} {{\rho _{1,r}}\prod\limits_{j = 2}^n {x_j^{r{\beta _{j1}}}} } \prod\limits_{j = n + 1}^m {x_j^{\left\lfloor {{{\left[ {r\beta _{j1}^{}} \right]}_ + }} \right\rfloor  + \left\lfloor {{{\left[ { - ({d_1} - r)\beta _{j1}^{}} \right]}_ + }} \right\rfloor }} )=x_1^{ - 1}(\sum\limits_{r = 0}^{{d_1}} {{P_{1,r}}\prod\limits_{j = 2}^n {x_j^{r{\beta _{j1}}}} } \Big),$$
where ${P_{1,r}} = {\rho _{1,r}}\prod\limits_{j = n + 1}^m {x_j^{\left\lfloor {{{\left[ {r\beta _{j1}^{}} \right]}_ + }} \right\rfloor  + \left\lfloor {{{\left[ { - ({d_1} - r)\beta _{j1}^{}} \right]}_ + }} \right\rfloor }}. $
	
Thus the first  Laurent monomial of $x_1^{(n)}$ is $x_1^{{-1}}\prod\limits_{j = 2}^n {x_j^{ {b_{j1}}}} P_{{1,d_1}}^{ }$.

Suppose that   the first  Laurent monomial of 	$x_{k}^{(n)}$ has the form $x_{k}^{ - 1}\prod\limits_{j = k+1 }^n {x_j^{{k_j}}} P_{k}$ for each  $k\leq i-1$.
When $k=i$, we have
\begin{eqnarray*}
			x_i^{(n)} &=& x_i^{ - 1}\Big(\sum\limits_{r = 0}^{{d_i}} {{\rho _{i,r}}\prod\limits_{j = i + 1}^n {x_j^{r{\beta _{ji}}}} \prod\limits_{j = 1}^{i - 1} {{{(x_j^{(n)})}^{ - r{\beta _{ji}}}}} }
			\prod\limits_{j = n + 1}^m {x_j^{\left\lfloor {{{\left[ {r\beta _{ji}^{(i - 1)}} \right]}_ + }} \right\rfloor  + \left\lfloor {{{\left[ { - ({d_i} - r)\beta _{ji}^{(i - 1)}} \right]}_ + }} \right\rfloor }} \Big)\\
			&=& x_i^{ - 1}\Big(\sum\limits_{r = 1}^{{d_i}} {{P_{i,r}}\prod\limits_{j = i + 1}^n {x_j^{r{\beta _{ji}}}} \prod\limits_{j = 1}^{i - 1} {{{(x_j^{(n)})}^{ - r{\beta _{ji}}}}} }  + {P_{i,0}}\Big),
	\end{eqnarray*}
where ${P_r} =\rho _{i,r} \prod\limits_{j = n + 1}^m {x_j^{\left\lfloor {{{\left[ {r\beta _{ji}^{(i - 1)}} \right]}_ + }} \right\rfloor  + \left\lfloor {{{\left[ { - ({d_i} - r)\beta _{ji}^{(i - 1)}} \right]}_ + }} \right\rfloor }} .$

We have the following two cases.
\begin{itemize}
\item[(1)] If there exists $t \in [1,i-1]$ such that  ${\beta _{ti}} \ne 0$,
	 then the first  Laurent monomial of ${(x_t^{(n)})^{-r{\beta _{ti}}}}$ is $x_{t}^{ r{\beta _{ti}}}\prod\limits_{j = t+1 }^n {x_j^{-r{\beta _{ti}}{k_j}}} P_{t}^{-r{\beta _{ti}}}.$ Thus	we obtain the first  Laurent monomial of $	x_i^{(n)}$ is $x_i^{{-1}}P_{i,0}.$
\item[(2)] If  ${\beta _{ti}} = 0$ for any $t \in [1,i-1]$, then		
$$x_i^{(n)} = x_i^{ - 1}(\sum\limits_{r = 1}^{{d_i}} {{\rho _{i,r}}\prod\limits_{j = i + 1}^n {x_j^{r{\beta _{ji}}}} } {P_{i,r}} + {P_{i,0}}).$$
\end{itemize}
Thus the first  Laurent monomial of $	x_i^{(n)}$ is $x_i^{-1}\prod\limits_{j = i + 1}^n {x_j^{ {b_{ji}}}} P_{i,{d_i}}.$
	
Therefore the first  Laurent monomial of $	x_i^{(n)}$ must have the form $x_i^{ - 1}\prod\limits_{j = i + 1}^n {x_j^{{k_j}}} P_i,$
and the conclusion holds.
\end{proof}

 \begin{theorem}\label{t3.6}
Let  ${\Sigma}=( {\widetilde{\mathbf{x}},\rho,{\widetilde B}})$ be an acyclic and coprime generalized seed. Then the projective standard monomials in
${x_1},x_1^{(n)}, \ldots ,{x_n},x_n^{(n)}$ form a $\mathbb{ZP}$-basis of $\mathcal{A}( {\Sigma})$.
\end{theorem}
\begin{proof}
We label projective standard monomials in
 ${x_1},x_1^{(n)}, \ldots ,{x_n},x_n^{(n)}$ by the points $\mathbf{a} = \left( {{a_1},{a_2}, \ldots {a_n}} \right) \in {{\mathbb{Z}}^n},$
  where ${\mathbf{x}^{\left\langle \mathbf{a} \right\rangle }} = x_1^{\left\langle {{a_1}} \right\rangle }x_2^{\left\langle {{a_2}} \right\rangle } \ldots x_n^{\left\langle {{a_n}} \right\rangle }$ and	
 	\[x_i^{\left\langle {{a_i}} \right\rangle } = \left\{ {\begin{array}{*{20}{c}}
 			{x_i^{{a_i}} ,   {\hspace{2.5cm}}if{\hspace{0.5cm}}{a_i} \ge 0};\\
 			{{{(x_i^{(n)})}^{ - {a_i}}},  {\hspace{1.5cm}}if{\hspace{0.5cm}}{a_i} < 0 }.
 \end{array}} \right.\]
	
 According to Lemma \ref{L3.5}, if $\mathbf{a}\prec\mathbf{a}'$, then the first monomial in $\mathbf{x}^{\left\langle \mathbf{a} \right\rangle} $ precedes the first monomial in ${\mathbf{x}^{\left\langle \mathbf{a}' \right\rangle }}$. Thus the projective
 standard monomials in ${x_1},x_1^{(n)}, \ldots ,{x_n},x_n^{(n)}$ are linearly independent over $\mathbb{ZP}$.	Therefore the proof follows from Theorem  \ref{thm1}.		
\end{proof}
\begin{remark}
\begin{itemize}
\item[(1)] If we set $d_i=1$ for all $i \in [1,n]$, one can obtain a $\mathbb{ZP}$-basis of the classical acyclic cluster algebra which is proved in \cite{BN}.
\item[(2)] For an acyclic cluster algebra, this basis should agree with an associated dual PBW basis in the language of \cite{KQ}. So we call the basis constructed in Theorem~\ref{t3.6} the dual PBW basis.
\end{itemize}
\end{remark}

\begin{example}
Consider the acyclic generalized seed $(\mathbf{\widetilde{x}},\rho,\widetilde{B})$ as follows:
\[ \mathbf{\widetilde{x}}=\{x_1,x_2,x_3,x_4,x_5,x_6,x_7\}, \rho=\{\rho_1,\rho_2, \rho_3\}\ and\ \widetilde B = \left( {\begin{array}{*{20}{c}}
 		0&{ - 1}&{ - 3}\\
 		2&0&{ - 3}\\
 		2&1&0\\
 		{ - 1}&2&1\\
 		1&{ - 1}&1\\
 		1&3&{ - 2}\\
 		{ - 2}&3&{ - 3}
 \end{array}} \right), \]
where  $\rho _1=\{1,l,1\}$, $\rho _2=\{1,1\}$ and $\rho_3=\{1,h,d,1\}$ for any  monomials $h,d,l\in\mathbb{Z}[x_{4},x_5,x_6, x_7]$.

According to the exchange relation, we have
\begin{eqnarray*}
 	{x'_1}& =& x_1^{ - 1}(x_4^{}x_7^2 + l{x_2}{x_3}{x_7} + x_2^2x_3^2x_5^{}x_6^{}),\\
 	{x'_2} &=& x_2^{ - 1}(x_1^{}x_5^{} + x_3^{}x_4^2x_6^3x_7^3),\\
 	{x'_3} &=& x_3^{ - 1}(x_1^3x_2^3x_6^2x_7^3 + hx_1^2x_2^2{x_6}x_7^2 + dx_1^{}x_2^{}x_7^{} + x_4^{}x_5^{}).
\end{eqnarray*}

By mutating the seed $(\mathbf{\widetilde{x}},\rho,\widetilde{B})$ in the direction $1$, we have
$${{\widetilde B}^{(1)}}= \left( {\begin{array}{*{20}{c}}
		0&1&3\\
		{ - 2}&0&{ - 3}\\
		{ - 2}&1&0\\
		1&1&{ - 2}\\
		{ - 1}&{ - 1}&1\\
		{ - 1}&3&{ - 2}\\
		2&1&{ - 9}
\end{array}} \right),$$ and
\[x_1^{(1)} = {x'_1} = x_1^{ - 1}(x_4^{}x_7^2 + k{x_2}{x_3}{x_7} + x_2^2x_3^2x_5^{}x_6^{}) = x_1^{ - 1}(x_4^{}x_7^2 + {A^{(1)}}),\]
where  ${A^{(1)}} = l{x_2}{x_3}{x_7} + x_2^2x_3^2x_5^{}x_6^{}.$

Continue to mutate the seed $({{{\bf{\widetilde x}}}^{(1)}},\rho^{(1)} ,{{\widetilde B}^{(1)}})$  in the direction $2$, we get
$${{\tilde B}^{(2)}} = \left( {\begin{array}{*{20}{c}}
		0&{ - 1}&3\\
		2&0&3\\
		{ - 2}&{ - 1}&0\\
		1&{ - 1}&{ - 2}\\
		{ - 3}&1&{ - 2}\\
		{ - 1}&{ - 3}&{ - 2}\\
		2&{ - 1}&{ - 9}
\end{array}} \right),$$ and
\[x_2^{(2)} = x_2^{ - 1}(x_1^{(1)}{x_3}{x_4}x_6^3{x_7} + {x_5}) = x_2^{ - 1}({A^{(2)}} + {x_5}),\]
where ${A^{(2)}} = x_1^{(1)}{x_3}{x_4}x_6^3{x_7}.$

By multiplying both sides of the above equation by $x_1$, we have	
\begin{eqnarray*}
	 && {x_1}x_2^{(2)} = x_2^{ - 1}({x_1}x_1^{(1)}{x_3}{x_4}x_6^3{x_7} + {x_1}{x_5})
	= x_2^{ - 1}((x_4^{}x_7^2 + {A^{(1)}}){x_3}{x_4}x_6^3{x_7} + {x_1}{x_5})\\
&	= &x_2^{ - 1}({x_3}x_4^2x_6^3x_7^3 + {x_1}{x_5}) + {f_1}
	= {x'_2} + {f_1},
\end{eqnarray*}	
where ${f_1}=x_2^{ - 1}{x_3}{x_4}x_6^3{x_7}{A^{(1)}} = lx_3^2{x_4}x_6^3{x^2_7} + x_2x_3^3{x_4}x_5^{}x_6^4{x_7}.$
	
Note that $ {f_1} \in \mathbb{Z} \mathbb{P}  \left[ {{x_1}, {x_2} ,{x_3}} \right],$  thus we obtain $x'_2\in \mathbb{ZP}  \left[ x_1, x_1^{(3)}, {x_2}, x_2^{(3)}, {x_3}, x_3^{(3)} \right]$.
	
Continue to mutate  the seed $({{{\bf{\widetilde x}}}^{(2)}},\rho^{(2)} ,{{\widetilde B}^{(2)}})$ in the  direction $3$, we get
\begin{eqnarray*}
		x_3^{(3)}& =& x_3^{ - 1}\big(x_4^2x_5^2x_6^2x_7^9 + hx_1^{(1)}x_2^{(2)}x_4^{}x_5^{}x_6^{}x_7^6 + d{(x_1^{(1)})^2}{(x_2^{(2)})^2}x_7^3
		+ {(x_1^{(1)})^3}{(x_2^{(2)})^3}\big)\\
		&=& x_3^{ - 1}\big(x_4^2x_5^2x_6^2x_7^9 + h(x_1^{ - 1}(x_4^{}x_7^2 + {A^{(1)}}))(x_2^{ - 1}({A^{(2)}} + {x_5}))x_4^{}x_5^{}x_6^{}x_7^6\\
		&&+ d{(x_1^{ - 1}(x_4^{}x_7^2 + {A^{(1)}}))^2}{(x_2^{ - 1}({A^{(2)}} + {x_5}))^2}x_7^3\\
		&&+ {(x_1^{ - 1}(x_4^{}x_7^2 + {A^{(1)}}))^3}{(x_2^{ - 1}({A^{(2)}} + {x_5}))^3}\big)\\
		&=& x_3^{ - 1}\big(x_4^2x_5^2x_6^2x_7^9 + hx_1^{ - 1}x_2^{ - 1}{(x_4^{}x_7^2 + {A^{(1)}})}{({A^{(2)}} + {x_5})}x_4^{}x_5^{}x_6^{}x_7^6\\
		&&+ dx_1^{ - 2}x_2^{ - 2}{(x_4^{}x_7^2 + {A^{(1)}})^2}{({A^{(2)}} + {x_5})^2}x_7^3\\
		&&+ x_1^{ - 3}x_2^{ - 3}{(x_4^{}x_7^2 + {A^{(1)}})^3}{({A^{(2)}} + {x_5})^3}\big)
\end{eqnarray*}	

By multiplying both sides of the above equation by $x_1^3x_2^3$, we obtain	
\begin{eqnarray*}
	x_1^3x_2^3x_3^{(3)}&=& x_3^{ - 1}\big(x_1^3x_2^3x_4^2x_5^2x_6^2x_7^9 + hx_1^2x_2^2(x_4^{}x_7^2 + {A^{(1)}})({A^{(2)}} + {x_5})x_4^{}x_5^{}x_6^{}x_7^6\\
	&&+ dx_1^{}x_2^{}{(x_4^{}x_7^2 + {A^{(1)}})^2}{({A^{(2)}} + {x_5})^2}x_7^3 + {(x_4^{}x_7^2 + {A^{(1)}})^3}{({A^{(2)}} + {x_5})^3}\big)\\
	&= &x_3^{ - 1}(x_1^3x_2^3x_4^2x_5^2x_6^2x_7^9 + hx_1^2x_2^2x_4^2x_5^2x_6^{}x_7^8 + dx_1^{}x_2^{}x_4^2x_5^2x_7^7 + x_4^3x_5^3x_7^6) + {f_2}\\
	&=&x_3^{ - 1}(x_1^3x_2^3x_6^2x_7^3 + hx_1^2x_2^2x_6^{}x_7^2 + dx_1^{}x_2^{}x_7^{} + x_4^{}x_5^{})x_4^2x_5^2x_7^6 + {f_2}\\
	&=& {x'_3}x_4^2x_5^2x_7^6 + {f_2}
\end{eqnarray*}		
where we denote  the remaining term by $f_2$.

It is easy to see that $f_2$ must have the following form
$$f_2=x^{-1}_3\big(\sum\limits_{1\leq i_1+i_2 \leq 6}g_{i_1,i_2}(x_1,x_2,x_3)({A^{(1)}})^{i_1}({A^{(2)}})^{i_2}\big),$$
where $g_{i_1,i_2}(x_1,x_2,x_3)\in \mathbb{ZP}\left[x_1,{x_2} ,{x_3} \right].$

Note that  both $x^{-1}_3A^{(1)}$ and $x^{-1}_3A^{(2)}$  belong to $\mathbb{ZP}   \left[ x_1, x_1^{(3)}, {x_2}, x_2^{(3)}, {x_3}, x_3^{(3)} \right]$, it follows that $ {f_2} \in \mathbb{ZP}  \left[ x_1, x_1^{(3)}, {x_2}, x_2^{(3)}, {x_3}, x_3^{(3)} \right].$ Thus we obtain ${x}'_3\in \mathbb{ZP}   \left[ x_1, x_1^{(3)}, {x_2}, x_2^{(3)}, {x_3}, x_3^{(3)} \right]$.
\end{example}

\section{A criteria of  generalized cluster algebras without 3-cycles}
Let  $( {\mathbf{\widetilde{x}},\rho,{\widetilde B}})$ be a generalized seed of a generalized cluster algebra of geometric type. We assume that
\[{b_{32}},{b_{21}}, {b_{13}} > 0.\]  Then in the associated directed graph $\Gamma\big(\widetilde{\mathbf{x}},\rho,\widetilde{B}\big)$ there is an oriented 3-cycle: $3 \to 2 \to  1   \to 3.$

According to the exchange relation, we have
\[\begin{array}{l}
	{x'_1} = \frac{{\sum\limits_{r = 0}^{{d_1}} {{P_{1,r}}x_2^{r {\beta _{21}}}x_3^{(r - {d_1}){\beta _{31}}}} }}{{{x_1}}},\\
	{x'_2} = \frac{{\sum\limits_{r = 0}^{{d_2}} {{P_{2,r}}x_1^{(r - {d_2}){\beta _{12}}}x_3^{r{\beta _{32}}}} }}{{{x_2}}},\\
	{x'_3} = \frac{{\sum\limits_{r = 0}^{{d_3}} {{P_{3,r}}x_1^{r{\beta _{13}}}x_2^{(r - {d_3}){\beta _{23}}}} }}{{{x_3}}},
\end{array}\]
where ${P_{i,r}} = {\rho _{i,r}}\prod\limits_{j = 4}^n {x_j^{r{{[{\beta _{ji}}]}_ + } + ({d_i} - r){{[ - {\beta _{ji}}]}_ + }}} \prod\limits_{j = n + 1}^m {x_j^{\left\lfloor {{{\left[ {r{\beta _{ji}}} \right]}_ + }} \right\rfloor  + \left\lfloor {{{\left[ {\left( {{d_i} - r} \right)\left( { - {\beta _{ji}}} \right)} \right]}_ + }} \right\rfloor }}$
for each $1\leq i\leq 3$.

Thus, we obtain
\[{x'_1}{x'_2}{x'_3} = \frac{{(\sum\limits_{r = 0}^{{d_1}} {{P_{1,r}}x_2^{r {\beta _{21}}}x_3^{(r - {d_1}){\beta _{31}}}} )(\sum\limits_{r = 0}^{{d_2}} {{P_{2,r}}x_1^{(r - {d_2}){\beta _{12}}}x_3^{r{\beta _{32}}}} )(\sum\limits_{r = 0}^{{d_3}} {{P_{3,r}}x_1^{r{\beta _{13}}}x_2^{(r - {d_3}){\beta _{23}}}} )}}{{{x_1}{x_2}{x_3}}}.\]

We denote
\begin{eqnarray}\label{(6.2)}
	[{r_1},{r_2},{r_3}]: &=&\frac{{{P_{1,{r_1}}}x_2^{{r_1}{\beta _{21}}}x_3^{({r_1} - {d_1}){\beta _{31}}}{P_{2,{r_2}}}x_1^{({r_2} - {d_2}){\beta _{12}}}x_3^{{r_2}{\beta _{32}}}{P_{3,{r_3}}}x_1^{{r_3}{\beta _{13}}}x_2^{({r_3} - {d_3}){\beta _{23}}}}}{{{x_1}{x_2}{x_3}}} \nonumber\\
	&=& \frac{{{P_{1,{r_1}}}{P_{2,{r_2}}}{P_{3,{r_3}}}x_1^{({r_2} - {d_2}){\beta _{12}} + {r_3}{\beta _{13}}}x_2^{({r_3} - {d_3}){\beta _{23}} + {r_1}{\beta _{21}}}x_3^{({r_1} - {d_1}){\beta _{31}} + {r_2}{\beta _{32}}}}}{{{x_1}{x_2}{x_3}}},
\end{eqnarray}
where  ${r_j}\in [0,{d_j}]$  for  each $1\leq j\leq 3$.

Then we get ${x'_1}{x'_2}  {x'_3} = \sum\limits_{{r_1}, {r_2},  {r_3}}^{} [{r_1} ,{r_2},  {r_3}]. $

\begin{theorem}\label{yheorem6.1}
With the above notations, if the standard monomials in ${{x_1},{x'_1}, \ldots ,{x_n},{x'_n}}$ are linearly independent
	over $\mathbb{ZP}$, then  the  directed graph $\Gamma\big(\widetilde{\mathbf{x}},\rho,\widetilde{B}\big)$ does not contain 3-cycles.
\end{theorem}
\begin{proof}
By Equation~(\ref{(6.2)}), we have
$$
{x'_1}{x'_2}{x'_3} = \sum\limits_{{r_1},{r_2},{r_3}} {[{r_1},{r_2},{r_3}]}
		= \sum\limits_{{r_2}} {[0,{r_2},{d_3}]}  + \sum\limits_{{r_3}} {[{d_1},0,{r_3}]}+\sum\limits_{{r_1}} {[{r_1},{d_2},0]}   +A,
$$
where $A = \sum\limits_{{r_1},{r_2},{r_3}} {[{r_1},{r_2},{r_3}]}  - \sum\limits_{{r_2}} {[0,{r_2},{d_3}]} -\sum\limits_{{r_3}} {[{d_1},0,{r_3}]}-\sum\limits_{{r_1}} {[{r_1},{d_2},0]}.$
	
Note that we have	
$$
\sum\limits_{{r_2}} {[0,{r_2},{d_3}]} =\frac{{\sum\limits_{{r_2} = 0}^{{d_2}} {{P_{1,0}}x_3^{ - {b_{31}}}{P_{2,{r_2}}}x_1^{({r_2} - {d_2}){\beta _{12}}}x_3^{{r_2}{\beta _{32}}}{P_{3,{d_3}}}x_1^{{b_{13}}}} }}{{{x_1}{x_2}{x_3}}}
		= {P_{1,0}}{P_{3,{d_3}}}x_1^{{b_{13}} - 1}{{x'}_2}x_3^{ - {b_{31}} - 1}.
$$
	
	Similarly, we have
	\begin{eqnarray*}
		\sum\limits_{{r_3}} {[{d_1},0,{r_3}]}  = {P_{1,{d_1}}}{P_{2,0}}x_1^{ - {b_{12}} - 1}x_2^{{b_{21}} - 1}{{x'}_3},\\
		\sum\limits_{{r_1}} {[{r_1},{d_2},0]}  = {P_{2,{d_2}}}{P_{3,0}}{{x'}_1}x_2^{ - {b_{23}} - 1}x_3^{{b_{32}} - 1}.
	\end{eqnarray*}
	
	Thus, we obtain that  $\sum\limits_{{r_2}} {[0,{r_2},{d_3}]}$, $\sum\limits_{{r_3}} {[{d_1},0,{r_3}]}$ and $\sum\limits_{{r_1}} {[{r_1},{d_2},0]}$ are all standard monomials in ${{x_1},{{x'}_1}, \ldots ,{x_n},{{x'}_n}}$ over $\mathbb{ZP}$.
	
	Note that  any monomial in $A$ has the following form
	\begin{eqnarray}\label{nonboth}
		{P_{1,{r_1}}}{P_{2,{r_2}}}{P_{3,{r_3}}}x_1^{({r_2} - {d_2}){\beta _{12}} + {r_3}{\beta _{13}} - 1}x_2^{({r_3} - {d_3}){\beta _{23}} + {r_1}{\beta _{21}} - 1}x_3^{({r_1} - {d_1}){\beta _{31}} + {r_2}{\beta _{32}} - 1},
	\end{eqnarray}
	where if  $r_1=0$, then  $r_3\ne d_3$; if  $r_2=0$, then  $r_1\ne d_1$; if  $r_3=0$, then $r_2\ne d_2$.
	
	We now consider the indices of $x_i$ $(1\leq i\leq 3)$ in (\ref{nonboth}). A simple argument shows that $$({r_3} - {d_3}){\beta _{23}} + {r_1}{\beta _{21}} - 1 \ge 0,$$
	$$({r_2} - {d_2}){\beta _{12}} + {r_3}{\beta _{13}} - 1 \ge 0,$$
	$$({r_1} - {d_1}){\beta _{31}} + {r_2}{\beta _{32}} - 1 \ge 0.$$
	
	Thus $A$ is a polynomial in ${{x_1}, \ldots ,{x_n}}$ over $\mathbb{ZP}$. Note that
	\begin{eqnarray*}
		&&{x'_1}{x'_2}{x'_3} \\
		&= &{P_{1,0}}{P_{3,{d_3}}}{{x'}_2}x_1^{{b_{13}} - 1}x_3^{ - {b_{31}} - 1}{ + P_{1,{d_1}}}{P_{2,0}}x_1^{ - {b_{12}} - 1}x_2^{{b_{21}} - 1}{{x'}_3}
		+ {P_{2,{d_2}}}{P_{3,0}}{{x'}_1}x_2^{ - {b_{23}} - 1}x_3^{{b_{32}} - 1}+ A.
	\end{eqnarray*}
It follows that the standard monomials in ${{x_1},{x'_1}, \ldots ,{x_n},{x'_n}}$ are linearly dependent
over $\mathbb{ZP}$, which is a contradiction. Hence the conclusion holds.
\end{proof}

\begin{remark}
In classical cluster algebras, Berenstein, Fomin and Zelevinsky  (\cite[Theorem 1.16]{bfz}) proved that standard monomials in $\{{x_1},{x'_1}, \ldots, {x_n}, {x'_n}\}$
are linearly independent over $\mathbb{ZP}$ if and only if the  directed graph associated to the seed $\big(\widetilde{\mathbf{x}},\widetilde{B}\big)$ is acyclic. The sufficient part was extended to the case of  generalized cluster algebras of geometric types by Bai, Chen, Ding and Xu (\cite[Theorem 3.1]{BCDX-1}).  However, the necessary part for generalized cluster algebras still remains open.
\end{remark}

\begin{example}
Consider the cyclic generalized seed $(\mathbf{\widetilde{x}},\rho,\widetilde{B})$ as follows:
\[\widetilde{\mathbf{x}}={\mathbf{x}}=\{x_1,x_2,x_3\},  \rho= \left\{{\rho _1},{\rho _2}, {\rho _3}\right\}\ and \ \widetilde{B} = B = \left( {\begin{array}{*{20}{c}}
			0&{ - 1}&3\\
			2&0&{ - 3}\\
			{ - 2}&1&0
	\end{array}} \right)\]
where ${\rho _1}=\left\{1,h,1\right\}$, ${\rho _2}=\left\{1,1\right\}$ and ${\rho _3}=\left\{1,d,l,1\right\}$ for any $h, d,l \in \mathbb{Z}.$

Note that  $3 \to 2   \to 1 \to 3$ is an oriented cycle associated to the  generalized seed $( \mathbf{\widetilde{x}},\rho,\widetilde{B}).$

According to the exchange relation, we have
\begin{eqnarray*}
	{x'_1} &=& \frac{{x_2^2 + hx_2^{}x_3^{} + x_3^2}}{{{x_1}}},\\
	{x'_2} &=& \frac{{x_1^{} + x_3^{}}}{{{x_2}}},\\
	{x'_3}& =& \frac{{x_1^3 + lx_1^2x_2 + dx_1x_2^2 + x_2^3}}{{{x_3}}}.
\end{eqnarray*}
A direct calculation shows that
\begin{eqnarray*}
	{x'_1}{x'_2}{x'_3} &=& {x'_1}x_2^2 + {x'_2}x_1^2x_3 + {x'_3}x_2^{} + (hl+1)x_1^2x_2+(l+hd)x_1x_2^{2}\\
	&&+ (h+l)x_1^2x_3+lx_1x_3^2+(hd+1)x_2^2x_3+dx_2x_3^2+(d+hl)x_1x_2x_3\\
    && +hx_1^3+(h+d)x_2^3.
\end{eqnarray*}
\end{example}
Thus the standard monomials in ${x_1},{x'_1}, {x_2},{x'_2} ,{x_3},{x'_3}$
are linearly dependent over $\mathbb{Z}$.
\section*{Acknowledgments}
Ming Ding was supported by NSF of China (No. 12371036) and Guangdong Basic and Applied Basic Research
Foundation (2023A1515011739) and Fan Xu was supported by NSF of China (No. 12031007).

\end{document}